\documentclass{amsart}
\usepackage{tikz}
\usepackage{pgfplots}
\pgfplotsset{width=10cm,compat=1.9}
\usepackage[export]{adjustbox}
\usepackage{caption}
\usepackage{subcaption}
\usepackage{wrapfig}
\usepackage{graphicx}
\usepackage{mathtools}

\usepackage{amssymb, latexsym, amsmath}%, makeidx}
\usepackage[all]{xy}

\usepackage{enumerate}
\usepackage{mathrsfs} %nice script font

\usepackage[OT2,T1]{fontenc}
\DeclareSymbolFont{cyrletters}{OT2}{wncyr}{m}{n}
\DeclareMathSymbol{\Sha}{\mathalpha}{cyrletters}{"58}

\theoremstyle{plain}
\newtheorem{Theorem}[subsection]{Theorem}

\newtheorem{Corollary}[subsection]{Corollary}

\newtheorem*{Conjecture*}{Conjecture}
\newtheorem{Proposition}[subsection]{Proposition}

\theoremstyle{definition}
\theoremstyle{remark}%This doesn't seem to work.

\newtheorem{Example}[subsection]{\bf Example}

\newtheorem{Remark}[subsection]{\bf Remark}

\numberwithin{equation}{subsubsection}

\renewcommand{\k}{{\kappa}}

\newcommand{\F}{{\mathbb F}}
\newcommand{\G}{{\text{\rm G}}}

\renewcommand{\H}{{\text{\rm H}}}

\newcommand{\N}{{\mathbb N}}
\renewcommand{\O}{{\text{\rm O}}}

\newcommand{\Q}{{\mathbb Q}}

\newcommand{\X}{{\text{\rm X}}}

\newcommand{\Z}{{\mathbb Z}}

\newcommand{\ab}{{\text{\rm ab}}}

\newcommand{\cor}{{\text{\rm cor}}}

\renewcommand{\gcd}{{\text{\rm gcd}}}

\newcommand{\ind}{{\text{\rm ind}}}

\newcommand{\inv}{{\text{\rm inv}}}
\newcommand{\isim}{{\;\overset{\sim}{\longrightarrow}\;}}
\newcommand{\isom}{{\;\simeq\;}}
\renewcommand{\ker}{{\text{\rm ker}}}
\newcommand{\lcm}{{\text{\rm lcm}}}

\newcommand{\per}{{\text{\rm per}}}

\newcommand{\res}{{\text{\rm res}}}

\newcommand{\Br}{{\text{\rm Br}}}

\newcommand{\Hom}{{\text{\rm Hom}}}

\newcommand{\Nm}{{\text{\rm N}}}

\begin{document}
\small

\title[Noncyclic Division Algebras]
{Noncyclic Division Algebras over Fields of\\ Brauer Dimension One}

\author{Eric Brussel}
\address 
{Department of Mathematics\\
California Polytechnic State University\\
San Luis Obispo, CA 93405\\ USA}
\email{ebrussel@calpoly.edu}

\begin{abstract}
Let $K$ be a complete discretely valued field of rank one, with residue field $\Q_p$.
It is well known that period equals index in $\Br(K)$.
We prove that when $p=2$ there exist noncyclic $K$-division algebras 
of every $2$-power degree divisible by four. Otherwise, every $K$-division algebra is cyclic.
\end{abstract}
\maketitle

\section{Introduction}
Recent research has focused on formulating and computing arithmetic invariants
in the Brauer group, over fields of cohomological dimension greater than $1$.
Two of the most prominent measures are {\it Brauer dimension} and {\it cyclic length}.
Famous fields of Brauer dimension one include local fields and global fields, 
and both have cyclic length one.
Because of the local-global principle and the degree-criterion for splitting
over local fields, one proves both ``cyclicity'' and ``period equals index'' for global fields at the same time.
Which raises the question: If $K$ is a field of 
Brauer dimension one, is it like a global field in this respect, that every $K$-division algebra is cyclic?
We show that the answer is no.

A field $K$ has {\it Brauer dimension one} if the Brauer group $\Br(K)$ is not trivial,
and every element in $\Br(K)$ of period $n$ is represented by a central simple algebra of degree $n$, for any $n$.
In other words, every class has equal period and index.
$K$ has {\it cyclic length one} if every class of period $n$ is represented by a {\it cyclic}
central simple algebra of degree $n$, for any $n$.
All $K$-division algebras are then cyclic crossed products.

When $K$ contains the $n$-th roots of unity, the theorem of Merkurjev-Suslin \cite{MS83} guarantees
that ${}_n\Br(K)$ is generated by cyclic classes of degree $n$,
and the cyclic length is called the {\it symbol length}.
This has been studied extensively;
see \cite{Am82} and \cite{ABGV} for a discussion of generation by cyclic classes, 
and \cite{Tig84} for fundamental results on symbol length.
Cyclic length is interesting not only as a measure of arithmetic complexity,
but also because of the central open problem in the area of finite-dimensional division algebras,
which is to determine whether every division algebra of prime degree is cyclic. 

Recent work on Brauer dimension %, also a measure of arithmetic complexity, 
includes \cite{Ant11}, \cite{AJ95}, \cite{Chip18}, \cite{Clark10}, \cite{dJ04}, \cite{HHK09}, \cite{Lieb11}, \cite{Lieb15}, 
\cite{PS10}, \cite{PS14}, and \cite{Sa97}.
Two-dimensional local fields, such as $\Q_p((t))$ and $\F_p((s))((t))$, 
were shown to have Brauer dimension one in \cite{Chip18},
and $\F_p((s))((t))$ was shown to have cyclic length one in \cite{AJ95}. %Corollary 3.5
The Brauer dimension of the function field of a $p$-adic curve
was proved to be two in \cite{Sa97}, and shown to have cyclic length two in \cite{BMT16}.
See \cite{BK00}, \cite{Ma16}, and \cite{Kra16} for known results on the relationship 
between Brauer dimension and symbol length. 

In this paper we show that the field $K=\Q_2((t))$
has Brauer dimension one and cyclic length two.
This is the first example of a field $K$ whose Brauer dimension is proved to be less than its cyclic length.
The testifying noncyclic division algebras are of every $2$-power degree divisible by four.
Noncyclic division algebras of equal period and degree themselves are not new, and were constructed
in the early 1930's by Albert, but over fields of Brauer dimension at least three \cite[Lemma 10]{Albert:deg4eg}.
Our examples exist in a very small space, 
in the sense that
they seem to depend on the noncyclicity of the cyclotomic extension $\Q_2(\mu_8)/\Q_2$,
they all have small cyclic splitting fields, and all $\Q_2((t))$-division algebras of odd degree are cyclic,
as are all $\Q_p((t))$-division algebras for odd $p$.
\section{Notation and Background}\label{background}

\subsection{For Fields of Characteristic Zero}
Let $K$ be a field of characteristic zero.
Let $\X(K)$ denote the character group, 
and $\Br(K)$ the Brauer group.
If $\delta\in\Br(K)$, the period $\per(\delta)$ is the order in $\Br(K)$,
and the index $\ind(\delta)$ is the smallest degree of a splitting field, equal
to the degree of the $K$-division algebra underlying $\delta$.
If $\delta\in\Br(K)$ and $E/K$ is a field extension, write $\delta_E$ for the restriction from $K$ to $E$.
If $\psi\in\X(K)$ has order $n$, 
let $\k(\psi)/K$ denote the cyclic field extension of degree $n$ fixed by $\ker(\psi)$.
If $E/K$ is a finite field extension, write $\psi_E$ for the restriction,
then $\k(\psi_E)$ is the corresponding cyclic extension of $E$, which is the compositum
of $\k(\psi)$ and $E$ inside some algebraic closure.

Let $\mu_n$ and $\mu$ denote the roots of unity groups, with
$\mu_n(K)$ and $\mu(K)$ the corresponding roots of unity in $K$.
Let $\zeta_n$ denote a primitive $n$-th root of unity, and put $i=\zeta_4$.
By Hilbert 90 we have an isomorphism $K^\times/K^{\times n}\isom\H^1(K,\mu_n)$;
let $(a)_n\in\H^1(K,\mu_n)$ denote either the class of $a\pmod{K^{\times n}}$,
or the cocycle $\G_K\to\mu_n$ defined by $\sigma\mapsto(a^{1/n})^{\sigma-1}$.
Since $-1\in K$ and $\mu_2$ is canonically identified with $\{\pm 1\}$,
we may identify $(a)_2$ with the character of order $2$.
We will cite {\it Albert's norm criterion},
which says a character $\psi$ extends to one of order $2|\psi|$ if and only
if $-1\in\Nm(\k(\psi)/K)$. This appears in \cite{AFSS}.

Let $(a,b)\in\Br(K)$ denote the standard {\it quaternion class},
for $a,b\in K$,
and put $(a,b)=1$ or $(a,b)=-1$, depending on whether the class is trivial or nontrivial.
A {\it cyclic class of degree $n$} is a Brauer class of the form $\theta\cup(b)_n$
where $\theta\in\X(K)$, $b\in K^\times$, and $|\theta|=n$. 
It then represents a cyclic central simple algebra of degree $n$. %Note period could be less than $n$.
If $\ind((\theta,b))=|\theta|$, the division algebra underlying $(\theta,b)$ is a cyclic crossed product,
with maximal subfield $\k(\theta)$.
By the norm criterion for cyclic classes, $(a,b)=1$ if and only if $b\in\Nm(K(\sqrt a)/K)$, and 
$(\theta,b)=0$ if and only if $b\in\Nm(\k(\theta)/K)$.

For $n\in\N$, the {\it $n$-cyclic length} of $K$ is zero if $\Br(K)$ has no elements of period $n$,
and otherwise it is the minimum number of cyclic classes of degree $n$ 
needed to express a general element of $\Br(K)$ of period $n$. We use $\infty$ if there is no such expression.
The {\it cyclic length} of $K$ is zero if $\Br(K)=\{0\}$, and otherwise it is the supremum of the $n$-cyclic lengths for all $n$.
The {\it $n$-Brauer dimension} of $K$ is zero if $\Br(K)$ has no elements of period $n$,
and otherwise it is the smallest number $d$ such that 
$\ind(\delta)$ divides $n^d$ for all $\delta\in\Br(K)$ of period $n$.
The {\it Brauer dimension} of $K$ is zero if $\Br(K)=\{0\}$, and otherwise it is the supremum 
of the $n$-Brauer dimensions over all $n$.
Since a class of period $n$ and $n$-cyclic length $d$ has index at most $n^d$,
the $n$-Brauer dimension is bounded by the $n$-cyclic length. 

\subsection{For Complete Discretely-Valued Fields}\label{dvfbackground}
Suppose $K$ is complete with respect to a discrete rank-one valuation, with residue field $k$
of characteristic zero, valuation ring $\O_K$, and uniformizer $t$.
Let $\mu_n^*=\Hom(\mu_n,\Q/\Z)$ be the Pontrjagin dual, then $\H^0(\mu_n,\Q/\Z)=(\mu_n^*)^{\G_k}$
is a subgroup of $\mu_n^*$ of order $|\mu_n(k)|$, and we have a split residue sequence
\[0\longrightarrow{}_n\X(k)\longrightarrow{}_n\X(K)\longrightarrow(\mu_n^*)^{\G_k}\longrightarrow 0\]
Thus every $\chi\in{}_n\X(K)$ is expressible as a sum $\chi=\psi+\omega^*\circ(t)_n$
for $\psi$ defined over $k$ and $\omega^*:\mu_n\to\Q/\Z$ of some order $e$ dividing $|\mu_n(k)|$. 
If $L=\k(e\chi)$ then $\chi_L=\omega^*\circ(vt)_n$ for some $v\in L$ by Kummer Theory, and then
$\k(\chi)=L(\root e\of{vt})$ has $t$-ramification degree $e$.

The split residue sequence 
\[0\longrightarrow\Br(k)\longrightarrow\Br(K)\longrightarrow\X(k)\longrightarrow 0\]
lets us express every $\delta\in\Br(K)$ in the form
$\delta=\alpha+(\theta,t)$, where $\alpha\in\Br(k)$ and $\theta\in\X(k)$ are identified with
their inflations to $K$.
We call this the {\it Witt decomposition} (with respect to $t$).
The {\it splitting criterion} states that
$\delta$ is split by a finite $t$-unramified field extension $E/k$ if and only if 
$E$ contains $\k(\theta)$ and splits $\alpha$, and
we have an {\it index formula} (see e.g. \cite{Br95})
\[\ind(\delta)=|\theta|\ind(\alpha_{\k(\theta)})\]

\subsection{For Local Fields}
Let $k$ be a local field of characteristic zero, and let $\zeta_m$ be a generator of $\mu(k)$.
The reciprocity map $\rho:k^\times\to\G_k^\ab$ induces an isomorphism \[\rho^*:\X(k)\isim(k^\times)^*\]
by Pontrjagin duality. For each $\theta\in\X(k)$ we set $\theta^*=\rho^*(\theta)$. 

If $k$ has residue field $\F_q$, Wedderburn's Theorem $\Br(\F_q)=0$
transforms the residue sequence into an isomorphism from $\Br(k)$ to $\X(\F_q)$,
and its composition with the evaluation at Frobenius yields the
{\it invariant map}
\[\inv_k:\Br(k)\isim\Q/\Z\]
Then \[\inv_k(\theta,b)=\theta^*(b)\] by \cite[Ch.XIV, Prop.3]{Se79} and Tate Duality.
In particular, $|\theta^*(b)|$ equals the order of $b$ in $k^\times/\Nm(\k(\theta)/k)$.

The decomposition $k^\times\isom\mu(k)\times\Z\times\Z_p^N$ induces a split exact sequence
\[0\longrightarrow(\Z\times\Z_p^N)^*\longrightarrow(k^\times)^*\longrightarrow\mu(k)^*\longrightarrow 0\]
Since $(\Z\times\Z_p^N)^*$ is divisible,
it follows that a character $\theta\in\X(k)$ of prime-power order extends to a character of any (larger)
prime-power order if and only if $\theta^*(\zeta_m)=0$, if and only if $\zeta_m\in\Nm(\k(\theta)/k)$.

\section{Results}

\begin{Proposition}
Let $K$ be a complete discretely valued field of rank one with residue field $k=\Q_p$.
Then $K$ has cyclic length at most two, and Brauer dimension one.
\end{Proposition}

\begin{proof}
This is well known. 
The Witt decomposition shows ${}_n\Br(K)$ is generated by $n$-cyclic classes, and its $n$-cyclic length is at most two.
Let $\delta=\alpha+(\theta,t)\in\Br(K)$.
By Local Class Field Theory (LCFT), $\ind(\alpha_{\k(\theta)})=\per(\alpha)/\gcd\{\per(\alpha),|\theta|\}$,
hence $\ind(\delta)=\per(\delta)=\lcm\{|\theta|,\per(\alpha)\}$ by the index formula.
Therefore $K$ has Brauer dimension one.
\end{proof}

\begin{Theorem}\label{main}
Let $K$ be a complete discretely valued field with residue field $\Q_2$,
Then there exist noncyclic $K$-division algebras of every $2$-power 
degree divisible by $4$.
\end{Theorem}

\begin{proof}
Set $k=\Q_2$, and let $t$ be a uniformizer for $K$.
Since $k$ does not contain a primitive fourth root of unity,
any character of $2$-power order over $K$ has residue of order dividing $2$,
and the corresponding cyclic extension has $t$-ramification degree dividing $2$.
Similarly, any tamely ramified cyclic field extension of $k$ of $2$-power degree 
has $p$-ramification degree dividing $2$.
Let $\alpha\in\Br(k)$ be an element of order $2^n$ divisible by $4$,
and let $D$ be the $K$-division algebra underlying the Brauer class
\[\delta=\alpha+(-1,t)\]
Then $\ind(D)=\per(\alpha)=2^n$.
$D$ has no $t$-unramified cyclic maximal subfields: Such a field extension would have degree $2^n$
and contain $k(i)$, by the splitting criterion, 
but $(-1,-1)=-1$ by computation of the Hilbert symbol (\cite[III.1.2, Theorem 1]{Se73}), hence $-1$
is not a norm from $k(i)$, hence $k(i)$ is not contained in a cyclic extension of degree $4$ by Albert's norm criterion.
Therefore no such subfield exists.
It follows that $D$ is cyclic if and only if it has a cyclic maximal subfield of $t$-ramification degree $2$.
Such fields have the form
\[E=L(\sqrt{vt})\]
where $L/K$ is of degree $2^{n-1}$ and defined over $k$, and $v\in\O_L^\times$.
Let $\bar L$ denote the residue field of $L$. Since the group of principal units of $\O_L$ is divisible, 
we may assume $v\in\bar L^\times$.
Let $\chi\in\X(K)$ be a character for $E/K$,
and let 
\[\psi=\chi-(t)_2\] 
Then $\k(\psi)=\bar L(\sqrt v)$ is a cyclic extension of $k$ of degree $2^n$.
Conversely if $\bar L(\sqrt v)/k$ is cyclic of degree $2^n$,
then $E=L(\sqrt{vt})$ is cyclic over $K$.
Therefore for our purposes $\bar L/k$ could be any cyclic extension of degree $2^{n-1}$
that extends to a cyclic extension of degree $2^n$.

Let $\tau=\sqrt{vt}$.
Compute \[\delta_E=\alpha_E+(-1,v^{-1}\tau^2)_E=\alpha_E+(-1,v^{-1})_E\]
Since $E/L$ is totally $t$-ramified, and $\alpha_E$ and $(-1,v^{-1})_E$ are defined over $\bar L$, 
the index of this class is the same as that of $\alpha_{\bar L}+(-1,v)_{\bar L}$
(see \cite[Lemma 1]{Br95}).
Therefore \[\ind(\delta_E)=\ind(\alpha_{\bar L}+(-1,v)_{\bar L})\]
Since $[\bar L:k]=2^{n-1}$, $\ind(\alpha_{\bar L})=2$ by the local theory.
Since $\Br(\bar L)\isom\Q/\Z$, any two classes of order $2$ are inverse to each other.
Therefore $D$ is cyclic via $E/K$ if and only if $(-1,v)_{\bar L}=-1$.

\noindent
{\it Claim}: $(-1,v)_{\bar L}=1$. 
It follows immediately that $\delta_E$ is nonzero,
and since $E/K$ is arbitrary, that $D$ is noncyclic.
By LCFT the corestriction from $\Br(\bar L)$ to $\Br(k)$
is injective, and since $-1\in k$, $\cor_{\bar L/k}(-1,v)_{\bar L}=(-1,\Nm_{\bar L/k}(v))$ by the cup product
formula for corestriction.
Therefore to prove the claim it suffices to show $(-1,\Nm_{\bar L/k}(v))=1$.

The claim applies to all possible $\bar L/k$, namely all
$\bar L/k$ of degree $2^{n-1}$ that extend to cyclic extensions of degree $2^n$,
and we must test all $\bar L$ and $v\in \bar L$ such that $\bar L(\sqrt v)/k$ is cyclic.
We first show that if the claim is true on a given $\bar L$ for one such $v$ then it is true for all $v$.
Suppose $\k(\psi)=\bar L(\sqrt v)$ is cyclic as above.
If $\psi'\in\X(k)$ is a character of order $2^n$ such that $2\psi'=2\psi$,
then $\psi'=\psi+(w)_2$ for some $w\in k$, and since $\psi_{\bar L}=(v)_2$, $\psi'_{\bar L}=(vw)_2$.
The corresponding $\chi'=\psi'+(t)_2\in\X(K)$ yields the cyclic extension $E'=L(\sqrt{vwt})$,
and $\delta_{E'}=(\alpha+(-1,t))_{E'}=\alpha_{E'}+(-1,vw)_{E'}$. 
Since $E'/L$ is totally ramified, the index of $\delta_{E'}$ can again be computed over $\bar L$,
and since $w\in k$,
$\cor_{\bar L/k}(-1,vw)_{\bar L}=(-1,\Nm_{\bar L/k}(v)w^{2^{n-1}})=(-1,\Nm_{\bar L/k}(v))$.
Thus if the claim is true for $\bar L$ and $v\in \bar L$ then $(-1,\Nm_{\bar L/k}(v))=1$ 
for any $v$.

We first prove the claim for $n=2$, i.e., $D$ of degree $4$. Then $[\bar L:k]=2$. We have
\[\H^1(\Q_2,\mu_2)=\{(\pm 1)_2, (\pm 2)_2, (\pm 5)_2, (\pm 10)_2\}\] by \cite[II.3.3, Corollary]{Se73},
hence $\Q_2$ has 7 nontrivial quadratic extensions.
By the computation of the Hilbert symbol (\cite[III.1.2, Theorem 1]{Se73}), $(-1,2)=(-1,5)=(-1,10)=1$ and $(-1,-1)=-1$,
and $(-1,-2)=(-1,-5)=(-1,-10)=-1$.
Therefore the candidates for $\bar L$ are
$\Q_2(\sqrt 2),\Q_2(\sqrt 5)$, or $\Q_2(\sqrt{10})$, by Albert's criterion.
Now:
\begin{enumerate}[\rm(i)]
\item
If $\bar L=\Q_2(\sqrt 2)$, let $v=2+\sqrt 2$.
Then $\Nm_{\bar L/\Q_2}(v)=2=\Nm_{\Q_2(i)/\Q_2}(1+i)$.
\item
If $\bar L=\Q_2(\sqrt 5)$, let $v=5+2\sqrt 5$.
Then $\Nm_{\bar L/\Q_2}(v)=5=\Nm_{\Q_2(i)/\Q_2}(2+i)$.
\item
If $\bar L=\Q_2(\sqrt{10})$, let $v=10+\sqrt{10}$.
Then $\Nm_{\bar L/\Q_2}(v)=90=\Nm_{\Q_2(i)/\Q_2}(9+3i)$.
\end{enumerate}
It is easy to check, using \cite[VI, Exercise 8]{L} for example, that each
$\bar L(\sqrt v)/\Q_2$ is cyclic of degree $4$.
We conclude that $(-1,\Nm_{\bar L/\Q_2}(v))=1$, hence $(-1,v)_{\bar L}=1$,
for all quadratic extensions $\bar L/\Q_2$
extendable to degree-$4$ cyclic extensions.
Therefore $D$ is noncyclic. 

For the general case, we will reduce to the $n=2$ case, as follows.
Let $\bar L'/k$ be the quadratic subextension of $\bar L$, 
and let $v'=\Nm_{\bar L/\bar L'}(v)$.
Then $v'$ is a product of the Galois conjugates of $v$ over $\bar L'$,
each of which can replace $v$ as an element whose square root generates $\bar L(\sqrt v)$, since $\bar L(\sqrt v)$
is a splitting field over $\bar L'$.
Since $\sqrt{v'}$ is the product of these square roots, it is contained in $\bar L$,
hence $\bar L'(\sqrt{v'})/k$ is a subextension of $\bar L/k$.

The element $v'$ is not a square in $\bar L'$, hence $\bar L'(\sqrt{v'})/k$ is cyclic of degree four.
To see this, let $\omega^*\in(\mu_n^*)^{\G_k}$ be the element of order $2$, and let $\res$ and $\cor$ be
restriction and corestriction for $\bar L/\bar L'$.
We have $(v')_{2,\bar L'}=\cor_{\bar L/\bar L'}((v)_{2,\bar L})$,
since the corestriction is induced by the norm, hence $\cor(\omega^*\circ(v)_{2,\bar L})=\omega^*\circ(v')_{2,\bar L'}$.
Since $\res(\psi_{\bar L'})=\omega^*\circ(v)_{2,\bar L}$, and
$\cor\circ\res$ is multiplication by $[\bar L:\bar L']=2^{n-2}$, we conclude
$\omega^*\circ(v')_{2,\bar L'}=2^{n-2}\psi_{\bar L'}$ 
is a character of order $2$, hence $v'$ is not a square in $\bar L'$.

Since $\bar L'(\sqrt{v'})/k$ is cyclic of degree $4$,
$(-1,v')=(-1,\Nm_{\bar L'/k}(v'))=1$ by the $n=2$ case of the claim.
But since $v'=\Nm_{\bar L/\bar L'}(v)$, we have
$\Nm_{\bar L'/k}(v')=\Nm_{\bar L/k}(v)$, so $(-1,\Nm_{\bar L/k}(v))=1$.
This proves the claim in general, and we conclude that $D$ is noncyclic in general.
\end{proof}

\begin{Remark}
It is not hard to show that 
the noncyclic division algebra of degree $2^n$ constructed above has a cyclic splitting field of degree $2^{n+1}$.
In fact, set $\zeta_j=\zeta_{2^j}$, and let $\bar E_j$, $j\geq 0$, be the subfield of $\Q_2(\zeta_{j+2})$ 
fixed by the automorphism $\zeta_{j+2}\mapsto\zeta_{j+2}^{-1}$.
Then $\bar E_j/\Q_2$ is cyclic of degree $2^j$, and $\bar E_j=\Q_2(\eta_j)$, 
where $\eta_j=\zeta_{j+2}+\zeta_{j+2}^{-1}$.
The sequence 
$\{\eta_j:j=0,\dots,\}$ is recursively definable by $\eta_0=0$ and $\eta_{j+1}^2=2+\eta_j$, and
it is easy to prove by induction that $\Nm_{\bar E_{j+1}/\Q_2}(\eta_{j+1})=2$ for $j\geq 1$.
Now $E=E_n(\sqrt{\eta_n t})$ is a cyclic
extension of $K$ of degree $2^{n+1}$, where $E_n=\bar E_n K$. 
Since $\Nm_{\bar E_n/\Q_2}(\eta_n)=2=\Nm_{\Q_2(i)/\Q_2}(1+i)$, we compute $(-1,\eta_n)=1$, 
and since $\bar E_n$ splits $\alpha$, we easily verify that $E$ splits $D$.
\end{Remark}

\begin{Theorem}\label{rest}
Let $K$ be a complete discretely valued field with residue field $\Q_p$,
and let $D$ be a $K$-division algebra.
Suppose either $\ind(D)$ is odd, or $p$ is odd.
Then $D$ is cyclic.
\end{Theorem}

\begin{proof}
Assume the hypotheses.
Set $k=\Q_p$, and let $m=|\mu(k)|$. 
Let $D$ be a $K$-division algebra of degree $n$, with Brauer class
$\delta=\alpha+(\theta,t)$ for $\alpha\in\Br(k)$ and $\theta\in\X(k)$.
If $\theta=0$ then $D$ is defined over the local field $k$, and is cyclic by local field theory.
Therefore suppose $\theta\neq 0$.
We may assume $n=\ell^a$ for some prime $\ell$ and $a\in\N$ by prime-power decomposition.
If $\k(\theta)$ extends to a ($t$-unramified) cyclic extension $E/k$ of degree $n$, then 
$E$ contains $\k(\theta)$ and splits $\alpha$, hence $D$ is cyclic via $E$ by the splitting principle.
Therefore we assume for the remainder of the proof that $1<|\theta|<n$, and $\theta$
does not extend to a character of order $n$.
Since $\ind(\delta)=n$ and $|\theta|<n$, it follows that $\ind(\alpha)=n$,
hence $\inv_k(\alpha)=i/n$ for some $i:\gcd(i,n)=1$.

Since it does not extend to a cyclic extension of degree $n$, $\k(\theta)/k$ is $p$-ramified, 
and $e=|\theta^*(\zeta_m)|$ is the smallest number such that $e\,\theta$ is
extendable to a character of any larger $\ell$-power order, hence
$e$ is the order of $\zeta_m$ in $k^\times/\Nm(\k(\theta)/k)$,
by Section~\ref{background}.
Evidently $e:1<e\leq|\theta|$.

By hypothesis, at least one of the numbers $\ell,p$ is odd, and
our assumption on the non-extendability of $\theta$ implies that $\ell\neq p$:
For if $\ell=p$, then $p$ is odd, so $K$ contains no $p$-th roots of unity, %If $p=2$ we have a $2$-th root of unity.
and the $p$-primary subgroup $\X(k)(p)$ is $p$-divisible by LCFT,
contradicting the non-extendability of $\theta$.
Therefore $\ell\neq p$.

Suppose $\ell\neq p$, and if $\ell=2$ then $p\equiv 1\pmod 4$.
We will construct a cyclic maximal subfield of $D$.
In this case $\k(\theta)/k$ is tame, $e$ is 
its ramification degree, and $\k(e\theta)$ is its maximal $p$-unramified subextension.
Since $\ell$ divides $e$ divides $m$,
and since $\zeta_m$ generates $\mu(k)$, $\zeta_m$ is not an $\ell$-th 
power in $k^\times$. If $\ell=2$ then $\zeta_m$
is not in $-4k^{\times 4}$ since $p\equiv 1\pmod 4$. 
Therefore $x^n-\zeta_m$ is irreducible by \cite[VI.9]{L}, with root $\zeta_{mn}$,
and the field $\bar L=k(\zeta_{mn/e})$ is cyclic and $p$-unramified of degree $n/e$ over $k$.
Since $\zeta_m$ has order $e$ modulo $\Nm(\k(\theta)/k)$, we can find a $j:\gcd(j,e)=1$ so that
$\inv_k(\theta,\zeta_m^j)=i/e$, where $i$ is as above.

Let 
\[v=\begin{cases}\zeta_{mn/e}^j&\text{if $\ell$ is odd}\\
\zeta_{mn/e}^{j(m/2+1)}&\text{if $\ell=2$}\end{cases}\]
Since the exponent is prime-to-$\ell$ in each case, $k(v)=k(\zeta_{mn/e})$.
This $v$ is designed to have norm $\zeta_m^j$:
When $\ell$ is odd, $v$ is a root of $x^{n/e}-\zeta_m^j$, so $\Nm_{k(v)/k}(v)=\zeta_m^j$.
When $\ell=2$, $v$ is a root of $x^{n/e}-\zeta_m^{j(m/2+1)}$, and since $4|m$, we compute 
$\Nm_{k(v)/k}(v)=-\zeta_m^{j(m/2+1)}=(-1)^{1+jm/2}\zeta_m^j=\zeta_m^j$.

Put $\bar L=k(v)$, $L=K(v)$.
Then $\bar L(\root e\of v)=k(\zeta_{mn})$ is cyclic of degree $n$ over $k$.
Let $\tau$ be a root of $x^e-vt\in L[t]$. Then we claim
\[E=L(\tau)\]
is a cyclic maximal subfield of $D$.
It is cyclic over $K$ as the cyclic extension of a sum $\phi+\omega^*\circ(t)_e$ as in Section~\ref{dvfbackground},
where $\phi\in\X(k)$ is such that $\k(\phi)=k(\zeta_{mn})/k$ and $\phi_{\bar L}=\omega^*\circ(v)_e$.
It remains to show $E$ is a maximal subfield of $D$.
Since $t=\tau^e/v$, 
\[\delta_E=\alpha_E+(\theta_E,\tau^e)+(\theta_E,1/v)\]
Since $\bar L/k$ is $p$-unramified, 
and since $\k(e\theta)$ is the maximal $p$-unramified subextension of $\k(\theta)/k$,
$|\theta_E|=|\theta_{\bar L}|=e$, so $(\theta_E,\tau^e)=0$, and
since $E/L$ is $t$-totally ramified, $\ind(\delta_E)=\ind(\alpha_{\bar L}+(\theta_{\bar L},1/v))$.

Since $[\bar L:k]=n/e$ and $\inv_k(\alpha)=i/n$, $\inv_{\bar L}(\alpha_{\bar L})=i/e$ by the local theory.
Since \[\cor_{\bar L/k}((\theta_{\bar L},1/v))=(\theta,\Nm_{\bar L/k}(1/v))=(\theta,\zeta_m^{-j})\]
by the cup product formula, and
since the corestriction is the identity on invariants by the local theory, 
we have $\inv_{\bar L}(\theta_{\bar L},1/v)=\inv_k(\theta,\zeta_m^{-j})=-i/e$.
Therefore $\inv_{\bar L}(\alpha_{\bar L}+(\theta_{\bar L},1/v))=i/e-i/e=0$, hence
$\delta_E=0$. Since $\ind(\delta)=[E:K]=n$, we conclude $E$ is a maximal subfield of $D$,
so $D$ is cyclic.

It remains to prove it for $\ell=2$ and $p\equiv 3\pmod 4$.
Set $\ind(D)=n=2^r$. If $r=1$ then $D$ is already known to be cyclic, so assume $r\geq 2$. %$r\geq 3$?
We have $\ind(\alpha)=2^r$ and $1<|\theta|<2^r$, and $\theta$ does not extend to a character of order $2^r$,
hence $\k(\theta)$ is $p$-ramified. Since $p\equiv 3\pmod 4$, all characters of $2$-power order
have residue of order at most $2$ by the sequence in Section~\ref{dvfbackground}, so
$\k(\theta)/k$ has ramification degree $2$, and  
$\theta$ does not extend to a character of larger $2$-power degree.
Therefore $-1\not\in\Nm(\k(\theta)/k)$ by Albert's criterion.

Let $s=v_2(p^2-1)$ be the (additive) $2$-value.
Then $\zeta_{2^s}$ is the largest $2$-power root of unity in $k(i)$.
Let $v=\zeta_{2^{s+r-2}}$, and set $\bar L=k(v)$, $L=K(v)$, and
\[E=L(\sqrt{vt})\]
By the local theory, $\bar L/k$ is cyclic unramified of degree $2^{r-1}$, $\theta_{\bar L}$ has order $2$,
and $\Nm_{\bar L/k(i)}(v)=-\zeta_{2^s}$. 
The Galois conjugate of $-\zeta_{2^s}$ over $k$ is $-\zeta_{2^s}^p$, the image under Frobenius, 
so $\Nm_{\bar L/k}(v)=\zeta_{2^s}^{p+1}$. We compute $v_2(p+1)=s-1$ since $p\equiv 3\pmod 4$,
therefore $\zeta_{2^s}^{p+1}$ has order $2$, therefore $\Nm_{\bar L/k}(v)=-1$.

An argument similar to the previous case shows $E/K$ is cyclic of degree $2^r$, and
\[\ind(\delta_E)=\ind(\alpha_{\bar L}+(\theta,v)_{\bar L})\]
Since $\Nm_{\bar L/k}(v)=-1$ is not a norm from $\k(\theta)$, $\ind((\theta,v)_{\bar L})=\ind((\theta,\Nm_{\bar L/k}(v)))\neq 1$,
and since $\theta_{\bar L}$ has order $2$, $\ind((\theta,v)_{\bar L})=2$.
Since $[\bar L:k]=2^{r-1}$ we compute
$\ind(\alpha_{\bar L})=2$ by the local theory, and since in $\Br(k)$ any two classes
of order $2$ are inverse to each other, we conclude $\alpha_{\bar L}+(\theta,v)_{\bar L}=0$.
Therefore $\ind(\delta_E)=1$,
so $D$ is split by $E$, hence $D$ is cyclic.
This completes the proof.
\end{proof}

\begin{Corollary}
Let $K$ be a complete discretely valued field with residue field $k=\Q_p$. 
Then the $2^r$-cyclic length of $K$ is two if $p=2$ and $4$ divides $2^r$;
the $n$-cyclic length of $K$ is one if $p=2$ and $n$ is odd or $n=2$;
and the cyclic length of $K$ is one if $p\neq 2$.
\end{Corollary}

\begin{proof}
Since all division algebras over the fields $\Q_2$ and $\Q_p$ are cyclic,
the Witt decomposition of ${}_{2^r}\Br(K)$ shows the $2^r$-cyclic length is at most two.
Theorem~\ref{main} shows the $2^r$-cyclic length is at least two for $\Q_2$ and $r\geq 2$.
If $2^r=2$ then the Witt decomposition shows that in all cases there is a quadratic splitting field,
hence the cyclic length is one.
The rest is Theorem~\ref{rest}.
\end{proof}

The {\it symbol length} does not detect the examples of Theorem~\ref{main}, because 
symbol length is only defined in circumstances where the examples don't exist:

\begin{Theorem}
Let $K$ be a complete, discretely valued field with residue field a local field $k$
of characteristic zero, and suppose $k$ contains $\mu_n$.
Then every $K$-division algebra of degree $n$ is cyclic.
\end{Theorem}

\begin{proof}
Let $t$ be a uniformizer for $K$.
As above, the Brauer class of a division algebra $D$ of degree $n$ has the form $\delta=\alpha+(\theta,t)$,
where $\alpha\in\Br(k)$ and $\theta\in\X(k)$ have order dividing $n$.
Since $k$ contains $\mu_n$, $\theta$ extends to a character $\psi$ of order $n$
by Kummer theory,
and the $t$-unramified cyclic extension $\k(\psi)/K$ then contains $\k(\theta)$ and splits $\alpha$, 
hence it splits $\delta$.
Therefore $D$ is cyclic.
\end{proof}

\begin{Example}
Let $K=k((t))$ for $k=\Q_3(i)$, and let $n$ be a power of a prime $\ell$.
We show how to prove every $K$-division algebra is cyclic.
We have $|\mu(k)|=8$, and 
so 
\[\H^1(k,\mu_n)\isom\begin{cases}
\Z/n\times\Z/\gcd(n,8)&\text{if $\ell\neq 3$}\\
\Z/n\times\Z/n\times\Z/n&\text{if $\ell=3$}
\end{cases}\]
Now arguments identical to those in the proof of Theorem~\ref{rest} show
that every $K$-division algebra is cyclic.
\end{Example}

\bibliographystyle{abbrv} %other choices are plain or abbrv or alpha
\bibliography{hnx.bib}

\end{document}